\documentclass[preprint,25pt]{elsarticle}
\usepackage[latin1]{inputenc}
\usepackage{amsmath}
\usepackage{amsfonts}
\usepackage{amssymb}
\usepackage{makeidx}
\usepackage{graphicx}
\usepackage{rotating}
\usepackage{tabularx}
\usepackage{wasysym}
\usepackage{tipa}
\usepackage{amsthm}
\usepackage{mathrsfs}
\usepackage{ragged2e}
\usepackage{hyperref}
\usepackage{csquotes}
\usepackage{natbib}
\usepackage{tikz}
\usepackage{easytable}
\usepackage{graphicx}
\usepackage{graphics}
\usepackage[all,cmtip]{xy}
\usepackage{epsfig}
\usepackage[toctitles]{titlesec}
\usepackage{amsfonts}
\usepackage{mathptmx}  
\usepackage{latexsym}
\usepackage{amssymb}
\usepackage{amsmath}
\usepackage{mathtools}
\usepackage{color}
\newtheorem{thm}{Theorem}

\newtheorem{definition}[thm]{\quad Definition}

\newtheorem{theorem}{Theorem}[section]
\newtheorem{lemma}[theorem]{Lemma}
\newtheorem{example}[theorem]{Example}

\newtheorem{remark}[theorem]{Remark}

\journal{....}

\usepackage{lscape}

\begin{document}

	\begin{frontmatter}

		\title{Sixteen-dimensional Sedenion-like Associative Algebra}
		\author{\large{Jitender$^1$,\hspace{0.2cm} Shiv Datt Kumar$^2$\vspace{.4cm}}}
		 \address{\large{$^{1, 2}$Department of Mathematics\\ Motilal Nehru National Institute of Technology Allahabad,\\ Prayagraj (UP), $211004$,
			India,}\\
			 Email: $^1$jitender2020rma03@mnnit.ac.in. $^2$sdt@mnnit.ac.in.}

		\begin{abstract}
			In this article, we construct a  $16$-dimensional sedenion-like associative algebra, which is  an even subalgebra of $2^5$-dimensional Clifford algebra $Cl_{5,0}$. We define the norm on sedenion-like algebra and show that its sixteen-dimensional elements preserves the norm relation $\lVert ST \rVert=\lVert S \rVert \lVert T \rVert$ under the condition $S_rS_d^\dagger + S_r^\dagger S_d=0$, where $S_r,~S_d$ denote the real and dual part of an octonion-like number $S$ respectively and $S^\dagger$ is the transpose of $S$. The elements of this sedenion-like algebra can be written as dual octonion like numbers called split bioctonion-like algebra and $S S^\dagger$ is commutative [i.e. $S S^\dagger=S^\dagger S $ and $(S S^\dagger) T=T(S S^\dagger )$], for any two octonion-like/sedenion-like numbers $S$ and $T$. We define the operations coproduct $\bigtriangleup$, counit $\epsilon $ and antipode $S$ on octonion-like/sedenion-like algebra to construct the Hopf algebra structure on it. We also show that $8$-dimensional octonion-like  associative seminormed division algebra is a $\mathbb{Z}_2^4/2$-graded quasialgebra and $16$ dimensional sedenion-like algebra  is a $\mathbb{Z}_2^5/2$-graded quasialgebra.
						
		\end{abstract}

	\end{frontmatter}
	
	\section{Introduction}
	 Christian \cite{CE} introduced an eight dimensional octonion like  associative normed division algebra over $\mathbb{R}$. Well known Hurwitz theorem states that  $\mathbb{R}$, $\mathbb{C}$, $\mathbb{H}$ and $\mathbb{O}$ are the only four normed division algebras over $\mathbb{R}$. Khalid and Bouchard \cite{BO} noted that octonion like algebra is a seminormed division algebra, i.e. division by any number $X$ in octonion like algebra  is possible if and only if $\lVert X \rVert \neq 0$. This octonion like algebra is different from the algebra of octonions $\mathbb{O}$ because it is associative and  has six imaginary and two real units where as the algebra of octonions $\mathbb{O}$ is non-associative and has seven imaginary and one real units.  
	 Jordan attempted to use the algebra of octonions to transfer the probabilistic interpretation of quantum theory, which is not successful due to  non-associativity of the octonions.  Hopf algebra \cite{SH} was introduced by Hopf which is related to the concept of $H$-space in Algebraic Topology. To see nonassociative algebras as an  associative algebra, Albuquerque and Majid \cite{AC} introduced the group graded quasialgebras, which are algebraic structures on the direct sum of homology and cohomology groups on an $H$-space in algebraic topology.  Linear Gr-categories \cite{HB} is the category of finite group graded vector space, which has applications in tensor categories \cite{BC}, quantum groups \cite{KQ}, cohomology of groups and  representation theory \cite{LA3}. Balodi et al. \cite{BD} proved that every expression in a G-graded quasialgebra can be reduced to a unique irreducible form and the irreducible words form a basis for the quasi-algebra, known as Diamond lemma.  
	 
	\justifying In this article, we generalize the octonion-like algebra and introduce  a  16 dimensional sedenion-like algebra which is an even subalgebra of $2^5$-dimensional Clifford algebra $Cl_{5,0}$ \cite{AC}. This algebra is different from the algebra of sedenion as it is associative and has six imaginary and ten real units, while the algebra of sedenions has one real and fifteen imaginary units. We split the 16-dimensional sedenion number as dual of octonion like numbers and define norm on it. Note that this algebra is a normed algebra but not a normed division algebra so it does not contradict the Hurwitz's theorem. We also prove that the sedenion-like elements $S,T$ preserves the condition $\lVert ST \rVert=\lVert S \rVert \lVert T \rVert$ and $S S^\dagger$ is commutative [i.e. $S S^\dagger=S^\dagger S$ and $(S S^\dagger) T=T(S S^\dagger )$]. Also, it does not give the $16$ square formula because in sedenion-like algebra we define the norm by using the condition $S_rS_d^\dagger+S_dS_r^\dagger=0$, which implies that  $s_0s_{15}=\sum_{i=1}^{i=7}s_is_{15-i}$ and $s_0s_8=\sum_{i=1}^{i=7} s_is_{8+i}$, where $S=s_0~+s_1\lambda e_oe_1~+s_2\lambda e_oe_2+s_3\lambda e_oe_3+s_4\lambda e_1e_2+s_5\lambda e_3e_1+s_6\lambda e_2e_3+s_7\lambda e_oe_1e_2e_3+s_8\lambda e_oe_4+s_9\lambda e_1e_4+s_{10}\lambda e_2e_4+s_{11}\lambda e_3e_4+s_{12}\lambda e_oe_2e_1e_4+s_{13}\lambda e_oe_1e_3e_4+s_{14}\lambda e_oe_3e_2e_4 +~s_{15}\lambda e_1e_2e_3e_4$ is a sedenion-like number. Hence all coefficients are algebraically dependent . Associativity of sedenion-like algebra is an important property for the physical application of algebras. Jordan attempted to use the algebra of octonions and sedenions to transfer the probabilistic interpretation of quantum theory in dimension 8 and 16 respectively, which  was not successful due to non-associativity of the octonions and sedenions. Due to the associativity of octonion-like and sedenion-like algebras, these can be used to transfer the probabilistic interpretation of the quantum theory in Jordan problem.

	The outline of the article is as follows: In section 2, an overview of $8$-dimensional octonion like associative algebra and their properties are presented. Sedenion like algebra and its properties are given in section 3. In section 4, norm on sedenion-like algebra is defined and proved that the given norm respects the condition  $\lVert XY \rVert=\lVert X \rVert \lVert Y \rVert$. Hopf algebra structure on sedenion-like/octonion-like algebras is given in Section 5, while Section 5 gives  quasialgebra structure on sedenion-like/octonion-like algebras.
	
	Throughout this article, $K$ denotes any field of characteristic zero, $K^*=K-\lbrace 0 \rbrace$, $\mathbb{O}^l$ is the octonion like associative seminormed division algebra, $\mathbb{S}^l$ is the sedenion-like associative normed algebra. 
	\section{Octonion-like Algebra}
	An octonion like algebra $\mathbb{O}^l$\cite{KO} is a seminormed division algebra over $\mathbb{R}$ (i.e. division by any $X\in \mathbb{O}^l$ is possible if $\lVert X\rVert \neq 0$ ), which is an even subalgebra of Clifford algebra $Cl_{4,0}$, i.e.
	\begin{center}
		$\mathbb{O}^l=span\lbrace 1,~ \lambda e_xe_y,~ \lambda e_ze_x,~ \lambda e_ye_z,~ \lambda e_xe_{\infty},~ \lambda e_ye_{\infty},~ \lambda e_ze_{\infty},~ \lambda e_xe_ye_ze_{\infty}\rbrace $
	\end{center}
The vectors $\lbrace e_x,~ e_y,~ e_z, ~e_{\infty}\rbrace $ form an orthonormal set of vectors which satisfy $e_i^2=1$ and $e_ie_j=-e_je_i$ ,~ for  $ i\neq j=\lbrace x,~y,~z,~{\infty}\rbrace$,  with the orientation $\lambda = 1$ or $\lambda= -1$. For simplicity writing $\lbrace 1,\lambda e_xe_y,\lambda e_ze_x,~ \lambda e_ye_z,~ \lambda e_xe_{\infty},~ \lambda e_ye_{\infty},~ \lambda e_ze_{\infty}, \lambda e_xe_ye_ze_{\infty}\rbrace$ as $\lbrace 1=u_0,\hspace{0.1cm} u_1, u_2,\hspace{0.1cm} u_3, u_4,\hspace{0.1cm} u_5,\hspace{0.1cm} u_6,\hspace{0.1cm} u_7\rbrace$ satisfying $u_i^2=1$, for $i=0, 7$ and $u_i^2=-1$ for $i=1, 2, 3, 4, 5, 6$ which implies that there are two real and six imaginary units, unlike the algebra of octonions, which has one real and seven imaginary units. The complete multiplication table is given by the following table:
 \begin{center}
	\begin{TAB}(r,0.4cm,0.4cm)[4pt]{|c|c|c|c|c|c|c|c|c|}{|c|c|c|c|c|c|c|c|c|}
	*&1& $u_1$ & $u_2$ & $u_3$ & $u_4$ & $u_5$ & $u_6$ & $u_7$\\
	
	1&1& $u_1$ & $u_2$ & $u_3$ & $u_4$ & $u_5$ & $u_6$ & $u_7$\\
	
	$u_1$&$u_1$& -$1$  & $u_3$  & $-u_2$ & $-u_5$ & $u_4$ & $u_7$ & $-u_6$ \\
	
	$u_2$& $u_2$& $-u_3$  & -$1$  & $u_1$ & $u_6$ & $u_7$ & $-u_4$ & $-u_5$  \\
	
	$u_3$& $u_3$& $u_2$  & $-u_1$  & $-1$ & $u_7$ & $-u_6$ & $u_5$ & $-u_4$ \\
	
	$u_4$& $u_4$& $u_5$  & $-u_6$  & $u_7$ & $-1$ & $-u_1$ & $u_2$ & $-u_3$ \\
	
	$u_5$& $u_5$& $-u_4$  & $u_7$  & $u_6$ & $u_1$ & $-1$ & $-u_3$ & $-u_2$ \\
	
	$u_6$& $u_6$& $u_7$  & $u_4$  & $-u_5$ & $-u_2$ & $u_3$ & $-1$ & $-u_1$ \\
	
	$u_7$& $u_7$& $-u_6$  & $-u_5$  & $-u_4$ & $-u_3$ & $-u_2$ & $-u_1$ & $1$ \\
\end{TAB}
\end{center}
The above multiplication table can also be obtained by the product $*$ defined by: 
\begin{center}
	$e_i*e_j=e_i\cdot e_j+e_i\wedge e_j.$
\end{center}
	This implies that $e_i\cdot e_j=\frac{1}{2}\lbrace e_i*e_j+e_j* e_i\rbrace$ and 
	$e_i\wedge e_j=\frac{1}{2}\lbrace e_i*e_j-e_j* e_i\rbrace$,
 where $\cdot$ is the scalar product and $\wedge$ is the wedge product. With unit vector $e_{\infty}$ the octonion-like algebra is an even subalgebra of sixteen dimensional Clifford algebra \cite{YO}.  So, this new algebra, which is different from the algebra of octonion is called as octonion-like algebra. Also octonion-like algebra is associative and noncommutative algebra. The octonion like associative algebra $\mathbb{O}^l$  is same as split bi-quaternion algebra (i.e. sum of two quaternion algebras) which can be presented as pair of quaternions same as the algebra of complex numbers can be written pair of real numbers. An element $X\in \mathbb{O}^l$ is written as the linear sum of all basis elements of $\mathbb{O}^l$, i.e. 
	\begin{center} 
		$X=x_ou_0+x_1u_1+x_2u_2+x_3u_3+x_4u_4+x_5u_5+x_6u_6+x_7u_7$
	\end{center}
 and $X^\dagger$ of an element $X\in \mathbb{O}^l$ is defined by changing the sign of coefficients of an  imaginary basis, i.e. 
	\begin{center}
		$X^\dagger=x_ou_0-x_1u_1-x_2u_2-x_3u_3-x_4u_4-x_5u_5-x_6u_6+x_7u_7$
	\end{center}
	Let $X, Y\in \mathbb{O}^l$, where
	\begin{center}
		$X=x_ou_0+x_1u_1+x_2u_2+x_3u_3+x_4u_4+x_5u_5+x_6u_6+x_7u_7$ and\\\vspace{0.2cm} $Y=y_ou_0+y_1u_1+y_2u_2+y_3u_3+y_4u_4+y_5u_5+y_6u_6+y_7u_7$. \\\vspace{0.2cm}
\hspace{-2cm}	Then \hspace{0.2cm}	$Z=XY=(x_ou_0+x_1u_1+x_2u_2+x_3u_3+x_4u_4+x_5u_5+x_6u_6+x_7u_7)$\vspace{0.2cm}\\\hspace{0.6cm}$(y_ou_0+y_1u_1+y_2u_2+y_3u_3+y_4u_4+y_5u_5+y_6u_6+y_7u_7)$\\\vspace{0.2cm}
		
		\hspace{1cm}	 $=
		(x_0y_0-x_1y_1-x_2y_2-x_3y_3-x_4y_4-x_5y_5-x_6y_6+x_7y_7)u_0$\vspace{0.2cm}\\
		\hspace{1cm}		 $+(x_1y_0+x_0y_1-x_3y_2+x_2y_3+x_5y_4-x_4y_5-x_7y_6-x_6y_7)u_1$\vspace{0.2cm}\\
		\hspace{1cm}	  $+(x_2y_0+x_3y_1+x_0y_2-x_1y_3-x_6y_4-x_7y_5+x_4y_6-x_5y_7)u_2$\vspace{0.2cm}\\
		\hspace{1cm}	   $+(x_3y_0-x_2y_1+x_1y_2+x_0y_3-x_7y_4+x_6y_5-x_5y_6-x_4y_7)u_3$\vspace{0.2cm}\\
		\hspace{1.2cm}	    $+(x_4y_0-x_5y_1+x_6y_2-x_7y_3+x_0y_4+x_1y_5-x_2y_6-x_3y_7)u_4$\vspace{0.2cm}\\
		\hspace{1cm}	     $+(x_5y_0+x_4y_1-x_7y_2-x_6y_3-x_1y_4+x_0y_5+x_3y_6-x_2y_7)u_5$\vspace{0.2cm}\\
		\hspace{1cm}	      $+(x_6y_0-x_7y_1-x_4y_2+x_5y_3+x_2y_4-x_3y_5+x_0y_6-x_1y_7)u_6$\vspace{0.2cm}\\
		\hspace{1cm}	       $+(x_7y_0+x_6y_1+x_5y_2+x_4y_3+x_3y_4+x_2y_5+x_1y_6+x_0y_7)u_7$ \\\vspace{0.2cm} \hspace{1cm} $=z_ou_0+z_1u_1+z_2u_2+z_3u_3+z_4u_4+z_5u_5+z_6u_6+z_7u_7$.
	\end{center}
	Also, the product $Z=XY$ can be written in matrix form as $Z_r=M_xY_r$, where $M_x$ is the $8\times 8$ matrix obtained by taking all the coefficients from left side in the above multiplication, given by:
	\begin{center}
		$M_x=\begin{bmatrix} 
			x_0 & -x_1 & -x_2 & -x_3 & -x_4 & -x_5 & -x_6 & x_7\\
			x_1 & x_0 & -x_3 & x_2 & x_5 & -x_4 & -x_7 & -x_6\\
			x_2 & x_3 & x_0 & -x_1 & -x_6 & -x_7 & +x_4 & -x_5\\
			x_3 &-x_2 & x_1 & x_0
			& -x_7 & x_6 & -x_5 & -x_4\\
			x_4 & -x_5& x_6 & -x_7 & x_0 & x_1 & -x_2 & -x_3\\
			x_5 & x_4 & -x_7 & -x_6 & -x_1 & +x_0 & +x_3 & -x_2\\
			x_6 & -x_7 & -x_4 & +x_5 & +x_2 & -x_3 & +x_0 & -x_1\\
			x_7 & x_6 & x_5 & x_4 & x_3 & x_2 & x_1 & x_0
		\end{bmatrix}$ \end{center}
	and $X_r$ and $Z_r$ are the real coefficients of $X$ and $Y$ written in matrix form as:
	\begin{center}
		$Y_r=\begin{bmatrix} 
			y_0 & y_1& y_2 & y_3& y_4 & y_5& y_6 & y_7
		\end{bmatrix}^T$ \end{center}
	\begin{center} 
		$Z_r=\begin{bmatrix} 
			z_0 & z_1& z_2 & z_3& z_4 & z_5& z_6 & z_7
		\end{bmatrix}^T$.
	\end{center} 
	Also, above product $Z=XY$ can be written as $Z=M_yX_r$, with  
	\begin{center}
		$M_y=\begin{bmatrix} 
			y_0 & -y_1 & -y_2 & -y_3 & -y_4 & -y_5 & -y_6 & y_7\\
			y_1 & y_0 & y_3 & -y_2 & -y_5 & y_4 & -y_7 & -y_6\\
			y_2 & -y_3 & y_0 & y_1 & y_6 & -y_7 & -y_4 & -y_5\\
			y_3 &y_2 & -y_1 & y_0 & -y_7 & -y_6 & y_5 &-y_4\\
			y_4 & y_5& -y_6 & -y_7 & y_0 & -y_1 & y_2 & -y_3\\
			y_5 & -y_4 & -y_7 & y_6 & y_1 & y_0 & -y_3 & -y_2\\
			y_6 & -y_7 & y_4 & -y_5 & -y_2 & y_3 & y_0 & -y_1\\
			y_7 & y_6 & y_5 & y_4 & y_3 & y_2 & y_1 & y_0
		\end{bmatrix}$ \end{center}
	and
	\begin{center}
		$X_r=\begin{bmatrix} 
			x_0 & x_1& x_2 & x_3& x_4 & x_5& x_6 & x_7
		\end{bmatrix}^T$. \end{center}
	In \cite{KO}, Khalid and Bouchard defined two seminorms $\lVert X\rVert_1$ and $\lVert X\rVert_2$  on octonion like algebra corresponding to $\lambda=1 $ and $\lambda=-1$ respectively given as follows:
	\begin{center}
		$\lVert X\rVert^2 =x_0x_0+x_1x_1+x_2x_2+x_3x_3+x_4x_4+x_5x_5+x_6x_6+x_7x_7+\lambda x_7x_0$\\\hspace{5cm}$-\lambda x_6x_1-\lambda x_5x_2-\lambda x_4x_3-\lambda x_3x_4-\lambda x_2x_5-\lambda x_1x_6+\lambda x_0x_7$
	\end{center}
		\[\hspace{-3.5cm}=\sum_{i=0}^{7}x_i^2-2\lambda\sum_{i=1}^{3}x_1x_{7-i}+2\lambda x_0x_7.\]
		The above $\lVert X\rVert_1$ and $\lVert X\rVert_2$ are seminorms but not norms as $\lVert 1-e_7\rVert_1=0$ but $ 1  - e_7 \textdoublebarslash 0 $, similarly  $\lVert 1+e_7\rVert_2 = 0$ but $1+e_7\neq 0$. In Theorem 1, we prove that $\mathbb{O}^l$ is seminormed division algebra and Theorem 2 proves that $XX^\dagger$ is commutative in $\mathbb{O}^l$.

	\begin{thm}
		An element $X\in \mathbb{O}^l$ has inverse in $\mathbb{O}^l$ if $\lVert X \rVert_1\neq 0$ and $\lVert X \rVert_2\neq 0$.
	\end{thm}
	\begin{proof}
		Let $X^{-1}= Y$ exists in $ \mathbb{O}^l$. Then $X Y = 1$ or in the matrix form, as defined above, it can be written as \begin{center}
			$M_x Y_r=\begin{bmatrix} 
				1 & 0& 0 & 0& 0 & 0& 0 & 0
			\end{bmatrix}^T$.\\ 
		\end{center}
		Hence inverse of $X\in \mathbb{O}^l$ exists only if matrix $M_x$ is non-singular. Now $M_x$ is non-singular if and only if  its all eigenvalues are nonzero. We calculate the eigenvalues of $M_x$  using the following Matlab code:
		\begin{center}
	\hspace{-2.5cm}	syms 	$x_0\hspace{0.2cm} x_1 \hspace{0.2cm} x_2\hspace{0.2cm} x_3\hspace{0.2cm} x_4\hspace{0.2cm} x_5\hspace{0.2cm} x_6\hspace{0.2cm} x_7$ 
		
	$M_x$ = [$x_0\hspace{0.2cm} -x_1 \hspace{0.2cm} -x_2\hspace{0.2cm}  -x_3\hspace{0.2cm}  -x_4\hspace{0.2cm}  -x_5\hspace{0.2cm}  -x_6\hspace{0.2cm}  x_7;$\\
		$x_1\hspace{0.2cm}  x_0 \hspace{0.2cm} -x_3\hspace{0.2cm}  x_2\hspace{0.2cm}   x_5 \hspace{0.2cm} -x_4\hspace{0.2cm}  -x_7\hspace{0.2cm}  -x_6;$\\
		$x_2\hspace{0.2cm}  x_3\hspace{0.2cm}  x_0 \hspace{0.2cm} -x_1\hspace{0.2cm}  -x_6\hspace{0.2cm}  -x_7\hspace{0.2cm}   x_4\hspace{0.2cm}  -x_5;$\\
		$x_3\hspace{0.2cm} -x_2\hspace{0.2cm}  x_1\hspace{0.2cm}   x_0 \hspace{0.2cm}  -x_7\hspace{0.2cm}  x_6\hspace{0.2cm}  -x_5\hspace{0.2cm}  -x_4;$\\
		$x_4\hspace{0.2cm} -x_5\hspace{0.2cm}  x_6\hspace{0.2cm}   -x_7\hspace{0.2cm}   x_0\hspace{0.2cm}  x_1\hspace{0.2cm}  -x_2\hspace{0.2cm}  -x_3;$\\
		
		$\hspace{0.8cm}x_5\hspace{0.2cm}  x_4\hspace{0.2cm} -x_7\hspace{0.2cm}   -x_6\hspace{0.2cm}  -x_1 \hspace{0.2cm} +x_0\hspace{0.2cm}  +x_3\hspace{0.2cm} -x_2;$\\
		
		$\hspace{1cm}x_6\hspace{0.2cm} -x_7\hspace{0.2cm} -x_4 \hspace{0.2cm}  +x_5\hspace{0.2cm}  +x_2\hspace{0.2cm}  -x_3\hspace{0.2cm}  +x_0\hspace{0.2cm} -x_1;$\\
		
		$\hspace{-1.5cm}x_7\hspace{0.2cm}  x_6\hspace{0.2cm}  x_5\hspace{0.2cm}   x_4 \hspace{0.2cm}   x_3\hspace{0.2cm}  x_2\hspace{0.2cm}    x_1\hspace{0.2cm}  x_0 $]\\
		\hspace{-5.5cm}eig($M_x$), 
	\end{center}
	which are given by 
		\begin{center}
			$\lambda_0=\lambda_1=x_0+x_7+i\sqrt{-x_1^2+2x_1x_6-x_2^2+2x_2x_5-x_3^2+2x_3x_4-x_4^2-x_5^2-x_6^2}$ \\\vspace{0.2cm}	$\lambda_2=\lambda_3=x_0+x_7-i\sqrt{-x_1^2+2x_1x_6-x_2^2+2x_2x_5-x_3^2+2x_3x_4-x_4^2-x_5^2-x_6^2}$ \\\vspace{0.2cm}	$\lambda_4=\lambda_5=x_0-x_7+i\sqrt{-x_1^2-2x_1x_6-x_2^2-2x_2x_5-x_3^2-2x_3x_4-x_4^2-x_5^2-x_6^2}$ \\\vspace{0.2cm}	$\lambda_6=\lambda_7=x_0-x_7-i\sqrt{-x_1^2-2x_1x_6-x_2^2-2x_2x_5-x_3^2-2x_3x_4-x_4^2-x_5^2-x_6^2}$ 
		\end{center}
		and the magnitude of eigenvalues is given by:
			\[|\lambda_0|= |\lambda_1|= |\lambda_2|= |\lambda_3|=\sum_{i=0}^{7}x_i^2-2\sum_{i=1}^{3}x_1x_{7-i}+2 x_0x_7= \lVert X\rVert_1\] 
			\[|\lambda_4|= |\lambda_5|= |\lambda_6|= |\lambda_7|=\sum_{i=0}^{7}x_i^2+2\sum_{i=1}^{3}x_1x_{7-i}-2 x_0x_7= \lVert X\rVert_2.\]
		It is given that  $\lVert X\rVert_1 \neq 0$ and $\lVert X\rVert_2 \neq 0$. Thus the magnitude of all eigenvalues are nonzero. Hence the inverse of  $X\in \mathbb{O}^l$ exists.  
	\end{proof}
\begin{thm}
	Let $X\in\mathbb{O}^l$. Then $XX^\dagger=X^\dagger X$ and $(XX^\dagger)Y=Y(XX^\dagger)$,\hspace{0.2cm} for all $Y\in\mathbb{O}^l$. 
\end{thm}
\begin{proof}
	Let  $X\in\mathbb{O}^l$. Then	\begin{center} 
		$X=x_ou_0+x_1u_1+x_2u_2+x_3u_3+x_4u_4+x_5u_5+x_6u_6+x_7u_7$,\\\vspace{0.2cm}\hspace{-0.2cm}
		 $X^\dagger=x_ou_0-x_1u_1-x_2u_2-x_3u_3-x_4u_4-x_5u_5-x_6u_6+x_7u_7$.\\\vspace{0.2cm}\hspace{-.65cm}
and	$XX^\dagger =(x_0x_0+x_1x_1+x_2x_2+x_3x_3+x_4x_4+x_5x_5+x_6x_6+x_7x_7)u_0$\vspace{0.1cm}\\
	\hspace{0.6cm}		 $+(x_1x_0-x_0x_1+x_3x_2-x_2x_3-x_5x_4+x_4x_5+x_7x_6-x_6x_7)u_1$\vspace{0.1cm}\\
	\hspace{0.6cm}	  $+(x_2x_0-x_3x_1-x_0x_2+x_1x_3+x_6x_4+x_7x_5-x_4x_6-x_5x_7)u_2$\vspace{0.1cm}\\
	\hspace{0.6cm}	   $+(x_3x_0+x_2x_1-x_1x_2-x_0x_3+x_7x_4-x_6x_5+x_5x_6-x_4x_7)u_3$\vspace{0.1cm}\\
	\hspace{0.6cm}	    $+(x_4x_0+x_5x_1-x_6x_2+x_7x_3-x_0x_4-x_1x_5+x_2x_6-x_3x_7)u_4$\vspace{0.1cm}\\
	\hspace{0.6cm}	     $+(x_5x_0-x_4x_1+x_7x_2+x_6x_3+x_1x_4-x_0x_5-x_3x_6-x_2x_7)u_5$\vspace{0.1cm}\\
	\hspace{0.6cm}	      $+(x_6x_0+x_7x_1+x_4x_2-x_5x_3-x_2x_4+x_3x_5-x_0x_6-x_1x_7)u_6$\vspace{0.1cm}\\
	\hspace{0.6cm}	       $+(x_7x_0-x_6x_1-x_5x_2-x_4x_3-x_3x_4-x_2x_5-x_1x_6+x_0x_7)u_7$
\end{center}
We note that all the coefficients of imaginary units are zero. Therefore,
	\begin{center}
		 $XX^\dagger=(x_0x_0+x_1x_1+x_2x_2+x_3x_3+x_4x_4+x_5x_5+x_6x_6+x_7x_7)$\vspace{0.1cm}\\
	\hspace{2cm} $+(x_7x_0-x_6x_1-x_5x_2-x_4x_3-x_3x_4-x_2x_5-x_1x_6+x_0x_7)u_7=X^\dagger X$
\end{center}
It is clear from the multiplication table that $u_0$ and $u_7$ commute with all units. Hence	 $(XX^\dagger)Y=Y(XX^\dagger)$.
\end{proof}

	\section{The 16-dimensional algebra $\mathbb{S}^L$}
		The sedenion like algebra $\mathbb{S}^l$ is a sixteen-dimensional even subalgebra of $2^5$-dimensional Clifford algebra $Cl_{5,0}$, i.e. 
	\begin{center}
		$\mathbb{S}^l$= span$\lbrace 1,\hspace{0.1cm} \lambda e_0e_1,\hspace{0.1cm} \lambda e_0e_2,\hspace{0.1cm} \lambda e_0e_3,\hspace{0.1cm} \lambda e_1e_2,\hspace{0.1cm} \lambda e_3e_1,\hspace{0.1cm} \lambda e_2e_3,\hspace{0.1cm} \lambda e_0e_1e_2e_3,\hspace{0.1cm} \lambda e_0e_4,$\\\vspace{0.4cm}$ \lambda e_1e_4,\hspace{0.1cm} \lambda e_2e_4,\hspace{0.1cm} \lambda e_3e_4, \hspace{0.1cm}\lambda e_0e_2e_1e_4,\hspace{0.1cm} \lambda e_0e_1e_3e_4,\hspace{0.1cm} \lambda e_0e_3e_2e_4,\hspace{0.1cm} \lambda e_1e_2e_3e_4 \rbrace $
	\end{center}
The set of vectors $\lbrace e_0,\hspace{0.1cm} e_1, \hspace{0.1cm}e_2,\hspace{0.1cm} e_3,\hspace{0.1cm} e_4\rbrace $ is orthonormal which satisfy  $e_ie_j=-e_je_i$, where  $i\neq j$, for all $i, j\in \lbrace 0,\hspace{0.1cm} 1,\hspace{0.1cm}2,\hspace{0.1cm}3,\hspace{0.1cm}4\rbrace$, with the orientation $\lambda = 1$ or $\lambda= -1$ and the geometric product is same as defined in section 2. The complete multiplication table of sedenion-like algebra is given by: 
\begin{landscape}
	
	\begin{TAB}(r,0.5cm,0.5cm)[5pt]{|c|c|c|c|c|c|c|c|c|c|c|c|c|c|c|c|}{|c|c|c|c|c|c|c|c|c|c|c|c|c|c|c|c|}
		* & $\lambda e_{01}$ &$\lambda e_{02}$ &$\lambda e_{03}$ &$\lambda e_{12}$ &$\lambda e_{31}$ &$\lambda e_{23}$ &$\lambda e_{0123}$ &$\lambda e_{04}$ &$\lambda e_{14}$ &$\lambda e_{24}$ &$\lambda e_{34}$ &$\lambda e_{0214}$&$\lambda e_{0134}$ &$\lambda e_{0324}$ &$\lambda e_{1234}$ \\
		
		$\lambda e_{01}$ & $ -1$ & $ -e_{12}$& $ e_{31}$& $ e_{02}$& $ -e_{03}$& $ e_{0123}$& $ -e_{23}$& $ -e_{14}$& $ e_{04}$& $ -e_{0214}$& $ e_{0134}$& $ e_{24}$& $ -e_{34}$& $ e_{1234}$& $ -e_{0324}$ \\
		
		$\lambda e_{02}$ & $ e_{12}$ & $ -1$& $ -e_{23}$& $ -e_{01}$& $ e_{0123}$& $ e_{03}$& $ -e_{31}$& $ -e_{24}$& $ e_{0214}$& $ e_{04}$& $ -e_{0324}$& $ -e_{14}$& $ e_{1234}$& $ e_{34}$& $ -e_{0134}$ \\
		
		$\lambda e_{03}$ & $ -e_{31}$ & $ e_{23}$& $ -1$& $ e_{0123}$& $ e_{01}$& $ -e_{02}$& $ -e_{12}$& $ -e_{34}$& $ -e_{0134}$& $ e_{0324}$& $ e_{04}$& $ e_{1234}$& $ e_{14}$& $ -e_{24}$& $ -e_{0214}$ \\
		
		$\lambda e_{12}$ & $ -e_{02}$ & $ e_{01}$& $e_{0123}$& $ -1$& $ e_{23}$& $ -e_{31}$& $ -e_{03}$& $ -e_{0214}$& $ -e_{24}$& $ e_{14}$& $ e_{1234}$& $ e_{04}$& $ e_{0324}$& $ -e_{0134}$& $ -e_{34}$ \\
		
		$\lambda e_{31}$ & $ e_{03}$ & $ e_{0123}$& $ -e_{01}$& $ -e_{23}$& $ -1$& $ e_{12}$& $ -e_{02}$& $ -e_{0134}$& $ e_{34}$& $ e_{1234}$& $ -e_{14}$& $ -e_{0324}$& $ e_{04}$& $ e_{0214}$& $ -e_{24}$ \\
		
		$\lambda e_{23}$ & $ e_{0123}$ & $ -e_{03}$& $ e_{02}$& $ e_{31}$& $ -e_{12}$& $-1$& $ e_{01}$& $ -e_{0324}$& $ e_{1234}$& $ -e_{34}$& $ e_{24}$& $ e_{0134}$& $ -e_{0214}$& $ e_{04}$& $ -e_{14}$ \\
		
		$\lambda e_{0123}$ & $ -e_{23}$ & $ -e_{31}$& $ -e_{12}$& $ -e_{03}$& $ -e_{02}$& $ -e_{01}$& $ 1$& $ -e_{1234}$& $ -e_{0324}$& $ -e_{0134}$& $ -e_{0214}$& $ -e_{34}$& $ -e_{24}$& $ -e_{14}$& $ -e_{04}$ \\
		
		$\lambda e_{04}$ & $ e_{14}$ & $ e_{24}$& $ e_{34}$& $ -e_{0214}$& $ -e_{0134}$& $ -e_{0324}$& $ e_{1234}$& $ -1$& $ -e_{01}$& $ -e_{02}$& $ -e_{03}$& $ e_{12}$& $ e_{31}$& $ e_{23}$& $ -e_{0123}$ \\
		
		$\lambda e_{14}$ & $ -e_{04}$ & $ e_{0214}$& $ -e_{0134}$& $ e_{24}$& $ -e_{34}$& $ e_{1234}$& $ e_{0324}$& $ e_{01}$& $ -1$& $ -e_{12}$& $ e_{31}$& $ -e_{02}$& $ e_{03}$& $ -e_{0123}$& $ -e_{23}$ \\
		
		$\lambda e_{24}$ & $ -e_{0214}$ & $ -e_{04}$& $ e_{0324}$& $ -e_{14}$& $ e_{1234}$& $ e_{34}$& $ e_{0134}$& $ e_{02}$& $ e_{12}$& $ -1$& $ -e_{23}$& $ e_{01}$& $ -e_{0123}$& $ -e_{03}$& $ -e_{31}$ \\
		
		$\lambda e_{34}$ & $ e_{0134}$ & $ -e_{0324}$& $ -e_{04}$& $ e_{1234}$& $ e_{14}$& $ -e_{24}$& $ e_{0214}$& $ e_{03}$& $ -e_{31}$& $ e_{23}$& $-1$& $ -e_{0123}$& $ -e_{01}$& $ e_{02}$& $ -e_{12}$ \\
		
		$\lambda e_{0214}$ & $ e_{24}$ & $ -e_{14}$& $ -e_{1234}$& $ e_{04}$& $ e_{0324}$& $ -e_{0134}$& $ e_{34}$& $ e_{12}$& $ -e_{02}$& $ e_{01}$& $ e_{0123}$& $ 1$& $ -e_{23}$& $ e_{31}$& $ -e_{03}$ \\
		
		$\lambda e_{0134}$ & $ -e_{34}$ & $ -e_{1234}$& $ e_{14}$& $ -e_{0324}$& $ e_{04}$& $ e_{0214}$& $ e_{24}$& $ e_{31}$& $ e_{03}$& $ e_{0123}$& $ -e_{01}$& $ e_{23}$& $ 1$& $ -e_{12}$& $ -e_{02}$ \\
		
		$\lambda e_{0324}$ & $ -e_{1234}$ & $ e_{34}$& $ -e_{24}$& $ e_{0134}$& $ -e_{0214}$& $ e_{04}$& $ e_{14}$& $ e_{23}$& $ e_{0123}$& $ -e_{03}$& $ e_{02}$& $ -e_{31}$& $ e_{12}$& $ 1$& $ -e_{01}$ \\
		
		$\lambda e_{1234}$ & $ e_{0324}$ & $ e_{0134}$& $ e_{0214}$& $ -e_{34}$& $ -e_{24}$& $ -e_{14}$& $ e_{04}$& $ e_{0123}$& $ -e_{23}$& $ -e_{31}$& $ -e_{12}$& $ e_{03}$& $ e_{02}$& $ e_{01}$& $ 1$ \\
	\end{TAB}
\begin{center}
Here $e_{ij}=e_ie_j$ and $e_{ijkl}=e_ie_je_ke_l$, for  $i,j,k,l\in\lbrace 0,1,2,3,4\rbrace $. 
\end{center}
\end{landscape}
It is evident from above multiplication table that the sedenion-like algebra is an associative, noncommutative and has six $(1,\lambda e_0e_1e_2e_3,\lambda e_oe_2e_1e_4,\lambda e_oe_1e_3e_4,\lambda e_oe_3e_2e_4,$\\$\lambda e_1e_2e_3e_4)$- real unit vectors and ten ($\lambda e_oe_1,\lambda e_oe_2,\lambda e_oe_3,\lambda e_1e_2,\lambda e_3e_1,\lambda e_2e_3,\lambda e_oe_4,$\\$\lambda e_1e_4,\lambda e_2e_4,\lambda e_3e_4$) imaginary unit vectors, while the algebra of sedenion is neither associative nor commutative and has one real and fifteen imaginary unit vectors. Hence this new algebra is called as sedenion-like algebra. Now we prove that $\mathbb{S}^L$ is closed under multiplication. Let 
\begin{center}
	$S=s_0+s_1\lambda e_oe_1+s_2\lambda e_oe_2+s_3\lambda e_oe_3+s_4\lambda e_1e_2+s_5\lambda e_3e_1+s_6\lambda e_2e_3+s_7\lambda e_oe_1e_2e_3$\\\vspace{0.3cm}$+s_8\lambda e_oe_4+s_9\lambda e_1e_4+s_{10}\lambda e_2e_4+s_{11}\lambda e_3e_4+s_{12}\lambda e_oe_2e_1e_4+s_{13}\lambda e_oe_1e_3e_4$\\\vspace{0.3cm}\hspace{-6cm}$+s_{14}\lambda e_oe_3e_2e_4+s_{15}\lambda e_1e_2e_3e_4$,
\end{center}
and 
\begin{center}
	$T=t_0+t_1\lambda e_oe_1+t_2\lambda e_oe_2+t_3\lambda e_oe_3+t_4\lambda e_1e_2+t_5\lambda e_3e_1+t_6\lambda e_2e_3+t_7\lambda e_oe_1e_2e_3$\\\vspace{0.3cm}$+t_8\lambda e_oe_4+t_9\lambda e_1e_4+t_{10}\lambda e_2e_4+t_{11}\lambda e_3e_4+t_{12}\lambda e_oe_2e_1e_4+t_{13}\lambda e_oe_1e_3e_4$\\\vspace{0.3cm}\hspace{-6cm}$+t_{14}\lambda e_oe_3e_2e_4+t_{15}\lambda e_1e_2e_3e_4$.
\end{center}
It is evident from the multiplication table that $\mathbb{S}^l$ is closed under multiplication, i.e. there exists $U=ST\in \mathbb{S}^L$ such that 
	\begin{center}
 $U=ST=u_0+u_1\lambda e_oe_1+u_2\lambda e_oe_2+u_3\lambda e_oe_3+u_4\lambda e_1e_2+u_5\lambda e_3e_1+u_6\lambda e_2e_3$\\\vspace{0.3cm}$+u_7\lambda e_oe_1e_2e_3+u_8\lambda e_oe_4+u_9\lambda e_1e_4+u_{10}\lambda e_2e_4+u_{11}\lambda e_3e_4+u_{12}\lambda e_oe_2e_1e_4$\\\vspace{0.3cm}\hspace{-3.5cm}$+u_{13}\lambda e_oe_1e_3e_4+u_{14}\lambda e_oe_3e_2e_4+u_{15}\lambda e_1e_2e_3e_4$.
 \end{center}
Now we show that every $S\in \mathbb{S}^L$ can be written as dual of octonion-like numbers, i.e.
\begin{center}
	$S=S_r+S_d\epsilon$, 
\end{center}  
where $S_r$ and $S_d$ are two octonion like numbers and $\epsilon=-e_1e_2e_3e_4$, such that $\epsilon^2=1$ and $\epsilon^\dagger=\epsilon$ ($\dagger$ is the reverse operation defined in \cite{DGC}, \cite{DGP} ),
\begin{center}
	$S_r=s_0+s_1\lambda e_oe_1+s_2\lambda e_oe_2+s_3\lambda e_oe_3+s_4\lambda e_1e_2+s_5\lambda e_3e_1+s_6\lambda e_2e_3+s_7\lambda e_oe_1e_2e_3$
\end{center}
	and
	\begin{center}
		$S_d=-s_{15}+s_{14}\lambda e_oe_1+s_{13}\lambda e_oe_2+s_{12}\lambda e_oe_3+s_{11}\lambda e_1e_2+s_{10}\lambda e_3e_1+s_9\lambda e_2e_3$\\\vspace{0.2cm}\hspace{-8cm}$+s_8\lambda e_oe_1e_2e_3$.
	\end{center} 
Clearly, 
\begin{center}
	$S_d\epsilon = s_8\lambda e_oe_4+s_9\lambda e_1e_4+s_{10}\lambda e_2e_4+s_{11}\lambda e_3e_4+s_{12}\lambda e_oe_2e_1e_4+s_{13}\lambda e_oe_1e_3e_4$\\\vspace{0.3cm}\hspace{-6.5cm}$+s_{14}\lambda e_oe_3e_2e_4+s_{15}\lambda e_1e_2e_3e_4 $.
\end{center}
 
\section{Norm on sedenion-like algebra }

In this section, we define a normed subalgebra of $\mathbb{S}^L$ called as  sedenion-like algebra $\mathbb{S}^l$ and show that this norm preserves the condition  $\lVert ST \rVert=\lVert S \rVert \lVert T \rVert$ and $S\cdot S^\dagger$ is commutative [i.e. $S\cdot S^\dagger=S^\dagger\cdot S $ and $(S\cdot S^\dagger)\cdot T=T(S\cdot S^\dagger )], ~~  \forall~~ S,~T\in \mathbb{S}^l$.

\justifying Christian \cite{CE} used the norm on $\mathbb{O}^l$ as $\lVert X\rVert=\sum_{i=0}^{7}x_i^2$, which is equal to the square\vspace{0.1cm}\\ root of $XX^\dagger$ with non scalar part set to zero of the geometric product $XX^\dagger $, for all $X\in \mathbb{O}^l$, i.e. 
  \begin{equation}
	\lVert X\rVert^2=X\cdot X^\dagger.
  \end{equation}
Extend this definition of norm on $\mathbb{S}^l$. If $S\in \mathbb{S}^l$, then norm of $S=\sum_{i=0}^{15}s_i^2$, which is equal to the square root of $SS^\dagger$ with non scalar of $SS^\dagger$ set to zero. Let $S=S_r+S_d\epsilon,
  ~S^\dagger=S_r^\dagger+S_d^\dagger\epsilon $. Then 
  \begin{center}
 $SS^\dagger=(S_r+S_d\epsilon)(S_r^\dagger+S_d^\dagger\epsilon)$
 \\\vspace{0.2cm}$\hspace{-0.1cm}=(S_rS_r^\dagger+S_dS_d^\dagger)+(S_rS_d^\dagger+S_dS_r^\dagger)\epsilon$
 \\\vspace{0.2cm}$\hspace{0.3cm}=(\lVert S_r \rVert^2+\lVert S_d \rVert^2)+(S_rS_d^\dagger+S_dS_r^\dagger)\epsilon$, 
\end{center}
which implies that $\lVert S\rVert^2=\lVert S_r\rVert^2+\lVert S_d\rVert^2$, if the octonion-like numbers $S_r$ and $S_d$ are orthogonal i.e. $S_rS_d^\dagger+S_dS_r^\dagger=0$. Then  we define $\mathbb{S}^l$ as a set of paired octonion-like numbers 
\begin{center}
	$\mathbb{S}^l=\lbrace S:=S_r+S_d\epsilon \hspace{0.1cm} |\hspace{0.1cm} \lVert S\rVert^2=\lVert S_r\rVert^2+\lVert S_d\rVert^2\rbrace$.
\end{center}
\begin{example}
	Let $S=1+e_0e_1e_2e_3+e_0e_4+e_1e_2e_3e_4$. Then 
	\begin{center}
		 \hspace{-3.65cm}$S_r=S_r^\dagger=1+e_0e_1e_2e_3$,\\\vspace{0.2cm}  $ \hspace{-0.7cm}S_d=S_d^\dagger=-1+e_0e_1e_2e_3$,
	 which implies that\\\vspace{0.2cm}
	 $\hspace{-1.5cm}S_rS_d^\dagger+S_dS_r^\dagger=(-1+e_0e_1e_2e_3-e_0e_1e_2e_3+1)$\\\vspace{0.2cm}$ \hspace{1cm}+(-1-e_0e_1e_2e_3+e_0e_1e_2e_3+1)=0$.
	\end{center}
\end{example}
Now in order to use the condition	$(S_rS_d^\dagger+S_dS_r^\dagger)\epsilon=0$, we must ensure that it is closed under multiplication. The proof is given in the next Lemma 3.
\begin{lemma}
	$\mathbb{S}^l=\lbrace S:=s_e+s_d\epsilon \hspace{0.1cm} |\hspace{0.1cm} \lVert S\rVert^2=\lVert S_r\rVert^2+\lVert S_d\rVert^2\rbrace$ is closed under multiplication.
\end{lemma}
\begin{proof}
	Let $S=S_r+S_d\epsilon$ and $T=T_r+T_d\epsilon$ are in $\mathbb{S}^l$. Then $(S_rS_d^\dagger+S_dS_r^\dagger)=0$ and $(T_rT_d^\dagger+T_dT_r^\dagger)=0$. Now we  prove that $(ST)_r(ST)_d^\dagger+(ST)_d(ST)_r^\dagger=0$. Here 
	\begin{center}
		\hspace{-2cm}$ST=(S_r+S_d\epsilon)(T_r+T_d\epsilon)$\\\vspace{0.2cm}\hspace{-0.15cm}
		$=(S_rT_r+S_dT_d)+(S_rT_d+S_dT_r)\epsilon$\\\vspace{0.1cm}\hspace{0.4cm}
		$(ST)_r=(S_rT_r+S_dT_d) $ and $(ST)_d=(S_rT_d+S_dT_r)$\\\vspace{0.1cm}\hspace{-1.1cm}
			$(ST)_r^\dagger=(S_rT_r+S_dT_d)^\dagger=T_r^\dagger S_r^\dagger+T_d^\dagger S_d^\dagger$ \\\vspace{0.2cm}\hspace{-1cm}
			$(ST)_d^\dagger=(S_rT_d+S_dT_r)^\dagger=T_d^\dagger S_r^\dagger+S_d^\dagger T_r^\dagger$.
	\end{center} 
Therefore, 
\begin{center}
	\hspace{-8cm}	$(ST)_r(ST)_d^\dagger+(ST)_d(ST)_r^\dagger $\\\vspace{0.2cm}\hspace{3cm} 
		$=(S_rT_r+S_dT_d)(T_d^\dagger S_r^\dagger+S_d^\dagger T_r^\dagger)+(S_rT_d+S_dT_r)(T_r^\dagger S_r^\dagger+T_d^\dagger S_d^\dagger)$\\\vspace{0.2cm}\hspace{2.9cm} 
		$=S_rT_rT_d^\dagger S_r^\dagger+S_rT_rS_d^\dagger T_r^\dagger
		+S_dT_dT_d^\dagger S_r^\dagger+S_dT_dS_d^\dagger T_r^\dagger+
		S_rT_dT_r^\dagger S_r^\dagger$\\\vspace{0.2cm}\hspace{-0.6cm} 
		$+S_rT_dT_d^\dagger S_d^\dagger+
		S_dT_rT_r^\dagger S_r^\dagger+S_dT_rT_d^\dagger S_d^\dagger$ \\\vspace{0.2cm}\hspace{1.7cm} 
		$=(S_rT_rT_d^\dagger S_r^\dagger+S_rT_dT_r^\dagger S_r^\dagger)+(S_rT_rS_d^\dagger T_r^\dagger+S_dT_rT_r^\dagger S_r^\dagger)$\\\vspace{0.2cm}\hspace{1.8cm} 
		$+(S_dT_dT_d^\dagger S_r^\dagger+S_rT_dT_d^\dagger S_d^\dagger)+(S_dT_dS_d^\dagger T_r^\dagger+S_dT_rT_d^\dagger S_d^\dagger)$\\ \vspace{0.2cm}\hspace{0.4cm} 
		$=S_rS_r^\dagger(T_rT_d^\dagger +T_dT_r^\dagger)+T_rT_r^\dagger(S_rS_d^\dagger +S_d S_r^\dagger)$\\\vspace{0.2cm}\hspace{0.5cm} 
		$+T_dT_d^\dagger(S_d S_r^\dagger+S_r S_d^\dagger)+S_dS_d^\dagger (T_dT_r^\dagger +T_rT_d^\dagger)$\\
		\vspace{0.2cm}\hspace{1.4cm} 
		$\hspace{0.2cm}=0$,
		 for  $(S_rS_d^\dagger+S_dS_r^\dagger)=0$ and $(T_rT_d^\dagger+T_dT_r^\dagger)=0$.\end{center} 
	 Hence normed  $\mathbb{S}^l$ is closed under multiplication.
		\end{proof}

\begin{definition}
   	Let $G$ be a group, $KG$ be its group algebra and $A,B$ are two subsets of $G$. Then $(A,B)$ is called a multiplicative pair if it satisfies
   	\begin{center}
$\lVert ab \rVert=\lVert a \rVert \lVert b\rVert,\hspace{0.2cm}  \forall \hspace{0.2cm}a\in\hspace{0.1cm} $span$(A),\hspace{0.1cm} b\in$\hspace{0.1cm} span$(B)$. 		
   	\end{center}
\end{definition}
As an application, in order to find some multiplicative pair \cite{AQ} on sedenion-like algebra, it is very important to show that   $\lVert ST \rVert=\lVert S \rVert \lVert T\rVert$. 
\begin{thm}
	The norm defined in (1) satisfies  $\lVert ST \rVert=\lVert S \rVert \lVert T \rVert$, for all $S,T\in \mathbb{S}^l$.
\end{thm}
\begin{proof}
To prove this, it is convenient to write $S$ and $T$ as dual octonion numbers. Let  
\begin{center}
	$S=s_0+s_1\lambda e_oe_1+s_2\lambda e_oe_2+s_3\lambda e_oe_3+s_4\lambda e_1e_2+s_5\lambda e_3e_1+s_6\lambda e_2e_3+s_7\lambda e_oe_1e_2e_3$\\\vspace{0.3cm}$+s_8\lambda e_oe_4+s_9\lambda e_1e_4+s_{10}\lambda e_2e_4+s_{11}\lambda e_3e_4+s_{12}\lambda e_oe_2e_1e_4+s_{13}\lambda e_oe_1e_3e_4$\\\vspace{0.3cm}\hspace{-6cm}$+s_{14}\lambda e_oe_3e_2e_4+s_{15}\lambda e_1e_2e_3e_4$, and

	$T=t_0+t_1\lambda e_oe_1+t_2\lambda e_oe_2+t_3\lambda e_oe_3+t_4\lambda e_1e_2+t_5\lambda e_3e_1+t_6\lambda e_2e_3+t_7\lambda e_oe_1e_2e_3$\\\vspace{0.3cm}$+t_8\lambda e_oe_4+t_9\lambda e_1e_4+t_{10}\lambda e_2e_4+t_{11}\lambda e_3e_4+t_{12}\lambda e_oe_2e_1e_4+t_{13}\lambda e_oe_1e_3e_4$\\\vspace{0.3cm}\hspace{-6cm}$+t_{14}\lambda e_oe_3e_2e_4+t_{15}\lambda e_1e_2e_3e_4$.
\end{center}
Then $S=S_r+S_d\epsilon$  and $T=T_r+T_d\epsilon$,  where $S_r,S_d,T_r,T_d$ are octonion-like numbers given by
\begin{center}
	$S_r=s_0+s_1\lambda e_oe_1+s_2\lambda e_oe_2+s_3\lambda e_oe_3+s_4\lambda e_1e_2+s_5\lambda e_3e_1+s_6\lambda e_2e_3+s_7\lambda e_oe_1e_2e_3$,\vspace{0.2cm}\\
	\hspace{-1cm}$S_d=-s_{15}+s_{14}\lambda e_oe_1+s_{13}\lambda e_oe_2+s_{12}\lambda e_oe_3+s_{11}\lambda e_1e_2+s_{10}\lambda e_3e_1+s_9\lambda e_2e_3$\\\vspace{0.2cm}\hspace{-8cm}$+s_8\lambda e_oe_1e_2e_3$ and

	\hspace{-0.35cm}$T_r=t_0+t_1\lambda e_oe_1+t_2\lambda e_oe_2+t_3\lambda e_oe_3+t_4\lambda e_1e_2+t_5\lambda e_3e_1+t_6\lambda e_2e_3+t_7\lambda e_oe_1e_2e_3$,\vspace{0.2cm}\\
	\hspace{-0.48cm}$T_d=-t_{15}+t_{14}\lambda e_oe_1+t_{13}\lambda e_oe_2+t_{12}\lambda e_oe_3+t_{11}\lambda e_1e_2+t_{10}\lambda e_3e_1+t_9\lambda e_2e_3+t_8\lambda e_oe_1e_2e_3$.
	\end{center}

\begin{center}
\hspace{-3cm}	Therefore $ST=(S_r+S_d\epsilon)(T_r+T_d\epsilon)=(S_rT_r+S_dT_d)+(S_rT_d+S_dT_r)\epsilon$,\\\vspace{0.2cm}\hspace{-0.6cm}
	$(ST)^\dagger =((S_rT_r+S_dT_d)+(S_rT_d+S_dT_r)\epsilon)^\dagger =(S_rT_r+S_dT_d)^\dagger+(S_rT_d+S_dT_r)^\dagger\epsilon$\\\vspace{0.2cm}\hspace{-8.8cm}
	$\lVert ST\rVert^2 =(ST)(ST)^\dagger$\\\vspace{0.2cm}\hspace{0.5cm}
	$=\lbrace(S_rT_r+S_dT_d)+(S_rT_d+S_dT_r)\epsilon\rbrace \lbrace(S_rT_r+S_dT_d)^\dagger+(S_rT_d+S_dT_r)^\dagger\epsilon\rbrace$.
	\\\vspace{0.2cm}\hspace{-0.65cm}
	$=\lbrace(S_rT_r+S_dT_d)(S_rT_r+S_dT_d)^\dagger+(S_rT_d+S_dT_r)(S_rT_d+S_dT_r)^\dagger\rbrace$	\\\vspace{0.2cm}\hspace{-0.1cm}
	$+\lbrace(S_rT_r+S_dT_d)(S_rT_d+S_dT_r)^\dagger+(S_rT_d+S_dT_r)(S_rT_r+S_dT_d)^\dagger\rbrace\epsilon$
\end{center}
In the above sum the coefficient of $\epsilon $ is zero by Lemma 2. Hence
\begin{center}
		\hspace{-2cm}$\lVert ST\rVert^2 =\lbrace(S_rT_r+S_dT_d)(S_rT_r+S_dT_d)^\dagger+(S_rT_d+S_dT_r)(S_rT_d+S_dT_r)^\dagger\rbrace$	\\\vspace{0.2cm}\hspace{0.25cm}
	$=	S_rT_rT_r^\dagger S_r^\dagger+S_rT_rT_d^\dagger S_d^\dagger+S_d T_dT_r^\dagger S_r^\dagger+S_dT_dT_d^\dagger S_d^\dagger+ S_rT_dT_d^\dagger S_r^\dagger+S_rT_dT_r^\dagger S_d^\dagger $	\\\vspace{0.2cm}\hspace{-6.5cm}
	$+T_rS_dT_d^\dagger S_r^\dagger+T_rS_dT_d^\dagger S_d^\dagger $.
	
\end{center}
In the above sum, second and sixth terms are canceled as $(T_rT_d^\dagger+T_dT_r^\dagger)=0$ and third and seventh terms are canceled as $(S_rS_d^\dagger+S_dS_r^\dagger)=0$, which implies that 
\begin{center}
		$\lVert ST\rVert^2 = 	S_rT_rT_r^\dagger S_r^\dagger+S_dT_dT_d^\dagger S_d^\dagger+ S_rT_dT_d^\dagger S_r^\dagger+T_rS_dT_d^\dagger S_d^\dagger$\\\vspace{0.2cm}\hspace{2.5cm}
		$=\lVert S_r\rVert^2\lVert T_r\rVert^2+\lVert S_d\rVert^2\lVert T_d\rVert^2+\lVert S_r\rVert^2\lVert T_d\rVert^2+\lVert S_d\rVert^2\lVert T_r\rVert^2$.
\end{center} 
Hence  $\lVert ST \rVert=\lVert S \rVert \lVert T \rVert$.
\end{proof}
It is clear from the multiplication table of sedenion-like algebra that $SS^\dagger$ for all $S\in \mathbb{S}^l$ contains only real units, which are commutative. Hence $S S^\dagger=S^\dagger S$ and $(S S^\dagger) T=T(S S^\dagger )$. 
\section{Hopf Algebra structure on Octonion-like and sedenion-like Algebra}
In this section we construct Hopf algebra structure on octonion like algebra. A Hopf algebra is a six-tuple $(H, \mu, \eta, \bigtriangleup, \epsilon, S)$, where $H$ is a vector space over $K$, the product  $\mu:H\otimes H\rightarrow H$ and unit $\eta: H\rightarrow k $ satisfies the commutativity of the following diagrams:\\\vspace{0.3cm}
\begin{tikzpicture}[node distance={27mm}, thick, main/.style = {draw, circle}]
	\node (1) {$H\otimes H\otimes h$}; 	
	\node (2) [right of=1] {$H\otimes H$};
	\node (3) [below of=1] {$H\otimes H$};
	\node (4) [below of=2]{$H$};
	\draw[->] (1) -- node[midway, above left, sloped, pos=8/9] {$id \otimes \mu$ } (2);
	\draw[->] (1) -- node[midway, below left, sloped, pos=2/3] {$\mu \otimes id$} (3);
	\draw[->] (2) -- node[midway, above  , sloped, pos=1/2] {
			{$\mu$} } (4);
	\draw[->] (3) -- node[midway, above left , sloped, pos=1/2] {$\mu$} (4);
	
	\node (a) [right of=2]{$K\otimes H$}; 	
	\node (b) [right of=a] {$H\otimes H$};
	\node (c) [right of=b] {$K\otimes H$};
	\node (d) [below of=b]{$H$};
	\draw[->] (a) -- node[midway, above left, sloped, pos=1] {$\eta \otimes id$ } (b);
	\draw[->] (a) -- node[midway, below left, sloped, pos=2/3] {}(d);
	\draw[->] (b) -- node[midway, above , sloped, pos=1/2] {
			$\mu$ } (d);
	\draw[->] (c) -- node[midway, above , sloped, pos=1/2] {} (d);
	\draw[->] (c) -- node[midway, above , sloped, pos=1/2] {$id \otimes \eta $} (b);
\end{tikzpicture}\vspace{0.3cm}
The coproduct	$\bigtriangleup : H \rightarrow H \otimes H$ and counit $\epsilon: K \rightarrow H$ are such that the following diagrams commutes.

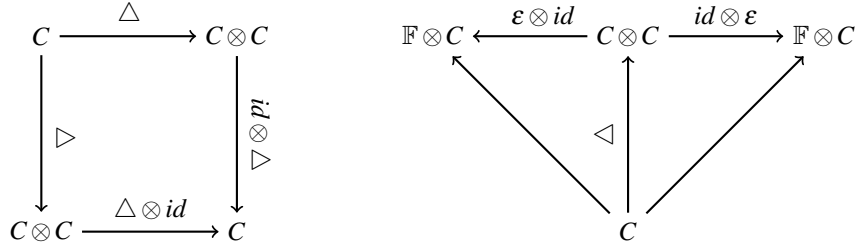
\begin{tikzpicture}[node distance={26mm}, thick, main/.style = {draw, circle}]
	\node (1) {$ C$}; 	
	\node (2) [right of=1] {$C\otimes C$};
	\node (3) [below of=1] {$C\otimes C$};
	\node (4) [below of=2]{$C$};
	\draw[->] (1) -- node[midway, above left, sloped, pos=2/3] {$\bigtriangleup$ } (2);
	\draw[->] (1) -- node[midway, above left, sloped, pos=2/3] {
			$\bigtriangleup $ } (3);
	\draw[->] (2) -- node[midway, above  , sloped, pos=1/2] {
			$id \otimes \bigtriangleup $ } (4);
	\draw[->] (3) -- node[midway, above , sloped, pos=1/2] {
		$\bigtriangleup \otimes id $} (4);
	\node (5) [right of=2]{$\mathbb{F}\otimes C$}; 	
	\node (6) [right of=5] {$C\otimes C$};
	\node (7) [right of=6] {$\mathbb{F}\otimes C$};
	\node (8) [below of=6]{$C$};
	\draw[->] (6) -- node[midway, above left, sloped, pos=1/40] {$\epsilon \otimes id$ } (5);
	\draw[->] (8) -- node[midway, below left, sloped, pos=2/3] {}(5);
	\draw[->] (8) -- node[midway, above , sloped, pos=1/2] {
			$\bigtriangleup $ } (6);
	\draw[->] (8) -- node[midway, above , sloped, pos=1/2] {} (7);
	\draw[->] (6) -- node[midway, above , sloped, pos=1/2] {$id \otimes \epsilon $} (7);
\end{tikzpicture}\\\vspace{0.3cm}
The antipode $S:H\rightarrow H$, satisfies the commutativity of following diagram:
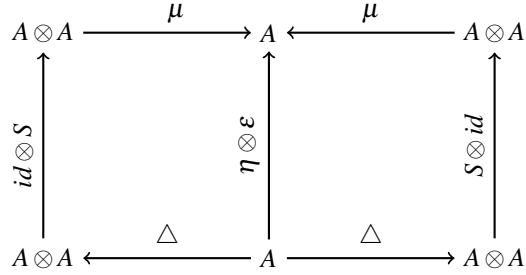
\vspace{0.3cm} 
\begin{center}
	\begin{tikzpicture}[node distance={30mm}, thick, main/.style = {draw, circle}]
		\node (1) {$A\otimes A$}; 	
		\node (2) [right of=1] {$A$};
		\node (3) [below of=1] {$A\otimes A$};
		\node (4) [below of=2]{$A$};
		\node (5) [right of=2] {$A\otimes A$};
		\node (6) [right of=4] {$A\otimes A$};
		\draw[->] (1) -- node[midway, above left, sloped, pos=2/3] {$\mu$ } (2);
		\draw[->] (3) -- node[midway, above left, sloped, pos=2/3] {$id \otimes S$} (1);
		\draw[->] (4) -- node[midway, above , sloped, pos=1/2] {$\eta \otimes \epsilon $} (2);
		\draw[->] (4) -- node[midway, above , sloped, pos=1/2] {$\bigtriangleup$} (3);
		\draw[->] (5) -- node[midway, above , sloped, pos=1/2] {$\mu$} (2);
		\draw[->] (4) -- node[midway, above , sloped, pos=1/2] {$\bigtriangleup$} (6);
		\draw[->] (6) -- node[midway, above , sloped, pos=1/2] {$S\otimes id$} (5);
	\end{tikzpicture}
\end{center}
Theorem 1 gives the existence of inverse of basis elements in  $\mathbb{O}^l$, which is used in the existence of antipode on $\mathbb{O}^l$.

	\begin{remark}
	All basis elements $u_i,\hspace{0.1cm} i=0,1,2,3,4,5,6,7$ in $\mathbb{O}^l$ have unique inverse $\mathbb{O}^l$.
	\end{remark}
\begin{proof}
	Since for all basis element $u_i, i=0,1,2,3,4,5,6,7, \lVert u_i\rVert_1=1$ and $\lVert u_i\rVert_2=1$. Therefore both norms are non-zero for all basis elements $u_i$, hence inverse of $u_i$ exists, by Theorem 1. 
\end{proof}
Now we are ready to define the Hopf algebra structure on octonion like algebra $\mathbb{O}^l$. The octonion-like algebras is closed under multiplication and has unit as given in the multiplication table.  Define the coproduct $\bigtriangleup(x):\mathbb{O}^l\rightarrow \mathbb{O}^l\otimes \mathbb{O}^l$ and  counit $\epsilon:\mathbb{O}^l\rightarrow \mathbb{R}$ by
\begin{center}
	$\bigtriangleup(x)=\bigtriangleup(\sum a_iu_i)=\sum a_i(u_i\otimes u_i)$, \hspace{1cm} $\epsilon(u_i)=1$,
\end{center}
If  $x =\sum a_iu_i\in \mathbb{O}^l$, then  
$(id\otimes \bigtriangleup)\bigtriangleup(x)=(id\otimes \bigtriangleup)\bigtriangleup(\sum a_iu_i )=(id\otimes \bigtriangleup)(\sum a_i(u_i\otimes u_i )) =\sum a_i(u_i\otimes( u_i\otimes u_i)) = a_i((u_i\otimes u_i)\otimes u_i) =(\bigtriangleup\otimes id)(a_i(u_i\otimes u_i))=(\bigtriangleup\otimes id)\bigtriangleup( \sum a_iu_i)=(\bigtriangleup\otimes id)\bigtriangleup(x)$, hence $ (id\otimes \bigtriangleup)\bigtriangleup(u_i)=(\bigtriangleup\otimes id)\bigtriangleup$, and $(\epsilon\otimes id)\bigtriangleup(x)=(\epsilon\otimes id)\bigtriangleup(\sum a_iu_i)(\epsilon\otimes id)(a_i(u_i\otimes u_i))=a_i(1\otimes u_i)$, similarly $(id\otimes \epsilon)\bigtriangleup(x)=(id\otimes \epsilon)\bigtriangleup(\sum a_iu_i)(id\otimes \epsilon)(a_i(u_i\otimes u_i))=a_i(u_i\otimes 1)$. Since $u_1\otimes 1=1\otimes u_1\in \mathbb{O}^l$ implies that $(\epsilon\otimes id)\bigtriangleup= (id\otimes \epsilon)\bigtriangleup$
which proves that it is a coalgebra. Define the antipode $S$ as
\begin{center}
	$S(x)=S(\sum a_iu_i)=\sum a_iu_i^{-1}$,\end{center}  for  all $x\in  \mathbb{O}^l$, where $x=\sum a_iu_i$ and $u_i^{-1}$ (inverse of $u_i$) exists by Corollary 3.1. Now we prove that $id*S=	S*id$ for basis elements of $\mathbb{O}^l$ and apply linearity to get the result for general elements. For any basis element $u_i\in \mathbb{O}^l$, 
\begin{center}
	$id*S(u_i)=\mu\circ(id\otimes S)\circ \bigtriangleup(u_i)=\mu\circ(id\otimes S)(u_i\otimes u_i)=\mu(u_i\otimes u_i^{-1})=1=\eta\circ\epsilon(u_i)$, \vspace{0.2cm}
	$S*id(u_i)=\mu\circ(S\otimes id)\circ \bigtriangleup(u_i)=\mu\circ(S\otimes id)(u_i\otimes u_i)=\mu(u_i^{-1}\otimes u_i)=1=\eta\circ\epsilon(u_i) $.
\end{center}  
 Thus $id*S=S*id=\eta\circ\epsilon$, hence $\mathbb{O}^l$ is a Hopf algebra. For sedenion-like algebra define $\bigtriangleup, \epsilon, S$ same as for octonion-like algebra. Here existence of inverse of basis element is given by: 
	$$
u_i^{-1}=\begin{cases}
	\hspace{0.3cm}u_i~; & \text{if $u_i$ is real unit }\\
	\hspace{0.3cm} u_i^3~; & \text{if $u_i$ is imaginary unit}
\end{cases}
$$   

\section{$\mathbb{Z}_2^n$-Graded Quasialgebra structure on Octonion-like and sedenion-like Algebra}
	Group graded quasialgebra \cite{HG} structure on octonion like algebras helps to find some multiplicative pairs on it. Therefore it is useful to construct the group graded quasialgebra structure on $\mathbb{O}^l$.  In this section, we give a quasi-algebraic structure on octonion like algebra $\mathbb{O}^l$. A group$(G)$-graded vector space means a vector space $V$ which can be written as
$V=\oplus_{g\in G}V_g$ and $V_g\cdot V_h\subseteq V_{gh}$ for $g, h\in G$, the elements of $ V_g $ are called homogeneous elements of degree g, denoted by $|v|$. A normalized $3$-cocycle on $G$ is a map $\phi:G\otimes G\otimes G\rightarrow K^*$ satisfying the following conditions:   
\begin{center}
	$\phi(ab,c,d)\phi(a,b,cd)=\phi(a,b,c)\phi(a,bc,d)\phi(b,c,d)$\\
	\vspace{0.2cm}
	$\hspace{0.35cm}\phi(a,e,b)=1,\hspace{0.2cm} \forall \hspace{0.2cm} a,b,c,d\in G $.
\end{center}
A $G$-graded quasialgebra is a $G$-graded vector space
$V$, a product map $V\otimes V\rightarrow V$ preserving the total degree and associative in the sense that
\begin{center}
	$(u\cdot v)\cdot w=u\cdot(v\cdot w)\phi(|u|,|v|,|w|),\hspace{0.1cm} \forall\hspace{0.1cm} u,v,w\in V$.
\end{center}
Let $x=(x_1,~x_2,~....~x_n)\in \mathbb{Z}_2^n$, define a group homomorphism $f:\mathbb{Z}_2^n\rightarrow \mathbb{Z}_2$ by $f(x)=\sum_{i=1}^{n} x_i$, then $Ker(f)$ forms a subgroup of $\mathbb{Z}_2^n$, denoted by $\mathbb{Z}_2^n/2$ and called the even subgroup of $\mathbb{Z}_2^n$. 
\begin{example}
	 Even subalgebra of $\mathbb{Z}_2^4$ is $\mathbb{Z}_2^4/2=\{(0,0,0,0),\hspace{0.2cm} (0,0,1,1),\hspace{0.2cm}(0,1,0,1),\vspace{0.1cm}$\\$ (0,1,1,0),\hspace{0.2cm}(1,0,0,1),\hspace{0.2cm}(1,0,1,0),\hspace{0.2cm}(1,1,0,0),\hspace{0.2cm}(1,1,1,1)\}$.  
Then  $\mathbb{Z}_2^4/2$-grading on\vspace{0.1cm} octonion like algebra is given by:
\begin{center} $|u_0|=(0,0,0,0);\hspace{1cm}|u_1|=(0,0,1,1)$;\\\vspace{0.2cm}$|u_2|=(0,1,0,1);\hspace{1cm}|u_3|=(0,1,1,0)$;\\\vspace{0.2cm}$|u_4|=(1,0,0,1);\hspace{1cm}|u_5|=(1,0,1,0)$;\\\vspace{0.2cm}$|u_6|=(1,1,0,0);\hspace{1cm}|u_7|=(1,1,1,1)$.
\end{center}
Clearly, the above grading satisfies that if $u_iu_j=u_k$, then $|u_i|+|u_j|=|u_k|$ for all $u_i,u_j,u_k\in \mathbb{O}^l$. Assuming that $\phi(|u_i|,|u_j|,|u_k|)=1$, we get  
\begin{center}
	$(u_iu_j)u_k=u_i(u_ju_k)\phi(|u_i|,|u_j|,|u_k|),\hspace{0.1cm} \forall\hspace{0.1cm} u,v,w\in V$.
\end{center}
\end{example} 
\begin{example}
	 $\mathbb{Z}_2^5/2$-grading quasialgebra structure on sedenion like algebra is given by:\begin{center}
 $|1|=(0,0,0,0,0);\hspace{3cm}|e_0e_1|=(0,0,0,1,1)$;
\vspace{0.2cm}\\$|e_0e_2|=(0,0,1,0,1);\hspace{2.6cm}|e_0e_3|=(0,1,0,0,1)$;
\vspace{0.2cm}\\$|e_1e_2|=(0,0,1,1,0);\hspace{2.6cm}|e_3e_1|=(0,1,0,1,0)$;
\vspace{0.2cm}\\$|e_0e_4|=(1,0,0,0,1);\hspace{2.6cm}|e_1e_4|=(1,0,0,1,0)$;
\vspace{0.2cm}\\$|e_2e_4|=(1,0,1,0,0);\hspace{2.7cm}|e_3e_4|=(1,1,0,0,0)$;
\vspace{0.2cm}\\\hspace{0.6cm}$|e_2e_3|=(0,1,1,0,0);\hspace{2.6cm}|e_0e_1e_2e_3|=(0,1,1,1,1)$;
\vspace{0.2cm}\\\hspace{0.6cm}$|e_0e_2e_1e_4|=(1,0,1,1,1);\hspace{2.1cm}|e_0e_1e_3e_4|=(1,1,0,1,1)$;
\vspace{0.2cm}\\\hspace{0.5cm}$|e_0e_3e_2e_4|=(1,1,1,0,1);\hspace{2.1cm}|e_1e_2e_3e_4|=(1,1,1,1,0)$,
\end{center}
 The associative of sedenion-like algebra implies that $\phi=1$, which preserves the condition  	$(u_iu_j)u_k=u_i(u_ju_k)\phi(|u_i|,|u_j|,|u_k|)$. Hence the sedenion-like algebra is a $\mathbb{Z}_2^5$-graded quasialgebra. 
\end{example}

	\vspace{0.2cm}\noindent\textbf{Acknowledgment}: 
	Authors thanks Prof. Hua-Lin Huang for stimulating discussion and useful comments. The first author sincerely thanks the Council of Scientific and Industrial Research (CSIR), Government of India for research fellowship (SRF: 09/1032 (0024)/2020-EMR-I).

\end{document}